\documentclass[a4paper]{article}
\usepackage[all]{xy}
\usepackage{amsmath, amssymb, amsthm}
\usepackage{latexsym, color}
\usepackage[dvipdfmx, hiresbb]{graphicx}
\usepackage{float}

\newtheorem{thm}{Theorem}[section]

\newtheorem{lem}[thm]{Lemma}

\newtheorem{cor}[thm]{Corollary}
\newtheorem{rem}[thm]{Remark}

\newtheorem{conj}[thm]{Conjecture}

\title{A resolution of Kellner's conjectures on Wilson quotients}
\author{Yutaro Matsuno}

\newcommand{\n}{\normalfont}
\newcommand{\ds}{\displaystyle}

\begin{document}

\maketitle

\begin{abstract} \n
Kellner expressed higher congruences for the Wilson quotient $W_p$ of an odd prime $p$ in terms of power sums of Fermat quotients, and recursively constructed certain $p$-independent polynomials $\psi_\nu$ occurring in these congruences. In this note, we give a simple $p$-adic proof of these congruences using the $p$-adic logarithm and $p$-adic exponential function. We also provide an explicit generating function for the polynomials $\psi_\nu$ in terms of complete Bell polynomials. This gives a direct derivation of Kellner's polynomials and allows us to resolve two conjectures stated by Kellner.
\end{abstract}

\section{Introduction}

Let $p$ be an odd prime. Wilson's theorem states that
\[
(p-1)! \equiv -1 \pmod p.
\]
This motivates the definition of the \textit{Wilson quotient}
\[
W_p=\frac{(p-1)!+1}{p}.
\]

On the other hand, Fermat's little theorem states that
\[
a^{p-1}\equiv 1 \pmod p
\]
for every integer $a$ prime to $p$. We define the \textit{Fermat quotient} by
\[
Q_p(a)=\frac{a^{p-1}-1}{p},
\]
and, for $n\ge1$, we write
\[
Q(p,n)=\sum_{a=1}^{p-1}Q_p(a)^n
\]
for the corresponding power sums.

Lerch \cite{Lerch} proved the congruence
\[
W_p\equiv Q(p,1)\pmod p.
\]
Kellner \cite{Kellner} recently generalized this interesting congruence to higher powers of $p$.
\\

\begin{thm}[Kellner] \n
There exist unique multivariate polynomials
\[
\psi_\nu(X_1,\dots,X_\nu)\in \mathbb{Z}[X_1,\dots,X_\nu]
~~(\nu\ge1),
\]
with no constant term, such that for every integer $n\ge1$ and every odd prime $p>n$,
\[
W_p\equiv\sum_{\nu=1}^n\frac{1}{\nu!}\psi_\nu(Q(p,1),\dots,Q(p,\nu))p^{\nu-1}\pmod {p^n}.
\]
\\
\end{thm}

Kellner stated two conjectures concerning the polynomials $\psi_\nu$.
\\

\begin{conj}[Kellner] \n
Let $\mathcal{P}(n)$ denote the partition function, and put
\[
\mathcal{P}_{\Sigma}(\nu)=\sum_{i=1}^{\nu}\mathcal{P}(i).
\]
Let $\mathbb{B}_n$ denote the complete Bell polynomial, and let
\[
\overline{H}_n=\sum_{i=1}^n\frac{(-1)^{i+1}}{i}
\]
be the alternating harmonic number. Then the following assertions hold:\\
(i) The polynomial $\psi_\nu$ has exactly $\mathcal P_{\Sigma}(\nu)$ terms.\\
(ii)\[
\psi_\nu(\{\pm1\})=-\mathbb{B}_\nu(\{\mp i!\overline{H}_i\}_{i=1,\dots,\nu}),
\]
and
\[
\sum_{\nu=1}^{\infty}\frac{1}{\nu!}\psi_\nu(\{\pm1\})T^\nu=1-(T+1)^{\mp\frac{1}{1-T}}.
\]
Here the signs are chosen correspondingly.
\\
\end{conj}

Kellner's construction of the polynomials $\psi_\nu$ is recursive. The aim of this note is to give a direct $p$-adic derivation of these congruences using the $p$-adic logarithm and exponential function, and to determine an explicit generating function for the polynomials $\psi_\nu$.
\\

\begin{thm} \n
Let
\[
H(\mathbf{X})(T):=-\exp\left(\sum_{i=1}^{\infty}\frac{(-1)^i}{i}X_i\frac{T^i}{1-T}\right)=-1+\sum_{\nu=1}^{\infty}\frac{1}{\nu!}\tilde{\psi}_\nu(\mathbf{X})T^\nu.
\]
Then
\begin{align*}
\tilde{\psi}_{\nu}(\mathbf{X})&=-\mathbb{B}_\nu\left(\left\{\sum_{i=1}^{k}\frac{(-1)^ik!}{i}X_i\right\}_{k=1,\dots,\nu}\right) \\
&=\sum_{\substack{m_1,\dots,m_{\nu}\ge 0 \\ 1\le m_1+2m_2+\dots+\nu m_{\nu}\le \nu}}(-1)^{m_1+2m_2+\dots+\nu m_{\nu}+1} \\
&~~~~~~~~~~~~\nu!\left(\begin{array}{c} \nu-(m_1+2m_2+\dots+\nu m_{\nu})+(m_1+\dots+m_{\nu})-1 \\ (m_1+\dots+m_{\nu})-1\end{array}\right) \\
&~~~~~~~~~~~~~~~~~~~~~~~~~~~~~~~~~~~~~~~~~~~~~~~~~~~~~~~~~~~~\left(\prod_{r=1}^{\nu} \frac{1}{r^{m_r}m_r!}\right)\left(\prod_{r=1}^{\nu}X_r^{m_r}\right),
\end{align*}
and
\[
\tilde{\psi}_\nu(\mathbf{X})\in \mathbb{Z}[X_1,\dots,X_{\nu}]
\]
for every $\nu\ge1$. Moreover, for every odd prime $p$, one has the $p$-adic identity
\[
W_p=\sum_{\nu=1}^{\infty}\frac{1}{\nu!}\tilde{\psi}_\nu(Q(p,1),\dots,Q(p,\nu))p^{\nu-1}.
\]
\\
\end{thm}

As a corollary, Kellner's higher congruences are recovered and generalized.
\\

\begin{cor} \n
For any odd prime $p$ and any natural number $n>0$, let
\[
m(n,p)=\max\left\{\nu=\sum_{i=0}^r a_ip^i\in\mathbb{N}_{>0}~\Bigg|~a_i\in \{0,\dots,p-1\},~\nu-1-\frac{1}{p-1}\left(\nu-\sum_{i=0}^r a_i\right)<n\right\}.
\]
Then,
\[
W_p\equiv \sum_{\nu=1}^{m(n,p)} \frac{1}{\nu!}\tilde{\psi}_{\nu}(Q(p,1),\dots,Q(p,\nu))p^{\nu-1}\pmod{p^n}.
\]
\end{cor}

\begin{proof} \n
Since $\tilde{\psi}_\nu$ has integral coefficients and $Q(p,1),\dots,Q(p,\nu)$ are integers, we have $\tilde{\psi}_\nu(Q(p,1),\dots,Q(p,\nu))\in \mathbb{Z}$. Therefore, by the well-known formula for the $p$-adic valuation of a factorial, if $\ds \nu=\sum_{i=0}^r a_i p^i~~(a_i\in\{0,\dots,p-1\})$, then
\[
v_p\left(\frac{1}{\nu!}\tilde{\psi}_\nu(Q(p,1),\dots,Q(p,\nu))p^{\nu-1}\right)\ge v_p\left(\frac{p^{\nu-1}}{\nu!}\right)=\nu-1-\frac{1}{p-1}\left(\nu-\sum_{i=0}^r a_i\right).
\]
Hence, if $\nu>m(n,p)$, then $\ds \frac{1}{\nu!}\tilde{\psi}_\nu(Q(p,1),\dots,Q(p,\nu))p^{\nu-1}\equiv 0 \pmod {p^n}$. Consequently,
\begin{align*}
W_p&=\sum_{\nu=1}^{\infty}\frac{1}{\nu!}\tilde{\psi}_\nu(Q(p,1),\dots,Q(p,\nu))p^{\nu-1} \\
&\equiv \sum_{\nu=1}^{m(n,p)}\frac{1}{\nu!}\tilde{\psi}_\nu(Q(p,1),\dots,Q(p,\nu))p^{\nu-1}\pmod {p^n}.
\end{align*}
\\
\end{proof}

\begin{rem} \n
If $p-1>n$, then $m(n,p)=n$. However, the above corollary does not fully recover Kellner's theorem, since in the case $n=p-1$ one has $m(n,p)=n+1$. This apparent obstruction can be removed by examining the polynomial $\tilde{\psi}_p$.
\\
\end{rem}

\begin{lem} \n
For any odd prime $p$,
\[
\tilde{\psi}_p(\mathbf{X})\equiv X_1^p-X_p\pmod{p}.
\]
In particular, for any odd prime $p$ and any positive integer $n$ with $p>n$,
\[
W_p\equiv \sum_{\nu=1}^n \frac{1}{\nu!}\tilde{\psi}_{\nu}(Q(p,1),\dots,Q(p,\nu))p^{\nu-1}\pmod{p^n}.
\]
\end{lem}

\begin{proof} \n
\begin{align*}
&\tilde{\psi}_p(\mathbf{X})=\sum_{\substack{m_1,\dots,m_p\ge 0 \\ 1\le m_1+2m_2+\dots+p m_p\le p}}(-1)^{m_1+2m_2+\dots+p m_p+1} \\
&~~~~~~~~~~~~p!\left(\begin{array}{c} p-(m_1+2m_2+\dots+p m_p)+(m_1+\dots+m_p)-1 \\ (m_1+\dots+m_p)-1\end{array}\right) \\
&~~~~~~~~~~~~~~~~~~~~~~~~~~~~~~~~~~~~~~~~~~~~~~~~~~~~~~~~~~~~\left(\prod_{r=1}^p \frac{1}{r^{m_r}m_r!}\right)\left(\prod_{r=1}^pX_r^{m_r}\right).
\end{align*}
If $(m_1,\dots,m_p)\neq (p,0,\dots,0),(0,\dots,0,1)$, then
\begin{align*}
&m_1,\dots,m_p<p,~~(m_1+\dots+m_p)-1<p, \\
&p-(m_1+2m_2+\dots+pm_p)<p,~~m_p=0.
\end{align*}
Therefore, the denominator of the coefficient
\[
p!\left(\begin{array}{c} p-(m_1+2m_2+\dots+p m_p)+(m_1+\dots+m_p)-1 \\ (m_1+\dots+m_p)-1\end{array}\right)\left(\prod_{r=1}^p \frac{1}{r^{m_r}m_r!}\right)
\]
is not divisible by $p$. Hence, 
\[
v_p\left\{p!\left(\begin{array}{c} p-(m_1+2m_2+\dots+p m_p)+(m_1+\dots+m_p)-1 \\ (m_1+\dots+m_p)-1\end{array}\right)\left(\prod_{r=1}^p \frac{1}{r^{m_r}m_r!}\right)\right\}\ge 1
\]
and then
\[
\tilde{\psi}_p(\mathbf{X})\equiv X_1^p-X_p\pmod{p}.
\]

Moreover, 
\begin{align*}
\tilde{\psi}_p(Q(p,1),\dots,Q(p,p))&\equiv Q(p,1)^p-Q(p,p) \\
&=\left(\sum_{a=1}^{p-1}Q_p(a)\right)^p-\sum_{a=1}^{p-1}Q_p(a)^p\equiv 0\pmod{p}.
\end{align*}
Then, 
\[
v_p\left(\frac{p^{p-1}}{p!}\tilde{\psi}_p(Q(p,1),\dots,Q(p,p))\right)\ge p-1
\]
and then
\begin{align*}
W_p&\equiv \sum_{\nu=1}^p \frac{1}{\nu!}\tilde{\psi}_{\nu}(Q(p,1),\dots,Q(p,\nu))p^{\nu-1} \\
&\equiv \sum_{\nu=1}^{p-1} \frac{1}{\nu!}\tilde{\psi}_{\nu}(Q(p,1),\dots,Q(p,\nu))p^{\nu-1}\pmod{p^{p-1}}.
\end{align*}
\\
\end{proof}

Consequently, the polynomials $\tilde{\psi}_\nu$ are Kellner's polynomials by the uniqueness assertion in Kellner's theorem.
\\

\begin{cor} \n
$\tilde{\psi}_{\nu}=\psi_{\nu}$ for all $\nu\ge 1$.
\\
\end{cor}

\section{Proof of main results}

\begin{proof}[Proof of Theorem 1.3] \n
By Wilson's theorem, we have $-(p-1)!\in 1+p\mathbb{Z}_p$. Hence
\[
(p-1)!=-\exp_p\bigl(\log_p(-(p-1)!)\bigr).
\]
The $p$-adic logarithm sends products to sums and kills roots of unity. Let $\omega$ be the Teichmüller character, and write
\[
a=\omega(a)\left<a\right>~~(a=1,\dots,p-1).
\]
Then
\[
\log_p(-(p-1)!)=\sum_{a=1}^{p-1}\log_p\left<a\right>.
\]
Since the Teichmüller character has order $p-1$, we obtain
\[
\log_p(-(p-1)!)=\frac{1}{p-1}\sum_{a=1}^{p-1}\log_p(a^{p-1})=-\frac{1}{1-p}\sum_{a=1}^{p-1}\log_p(1+pQ_p(a)).
\]
Since $1+pQ_p(a)\in 1+p\mathbb Z_p$, the $p$-adic logarithm is given by its usual power series. Therefore
\begin{align*}
\log_p(-(p-1)!)&=-\frac{1}{1-p}\sum_{a=1}^{p-1}\log_p(1+pQ_p(a)) \\
&=-\frac{1}{1-p}\sum_{a=1}^{p-1}\sum_{i=1}^{\infty}\frac{(-1)^{i-1}}{i}(pQ_p(a))^i \\
&=-\frac{1}{1-p}\sum_{i=1}^{\infty}\frac{(-1)^{i-1}}{i}Q(p,i)p^i.
\end{align*}

Now define
\[
G(\mathbf{X})(T)=-\frac{1}{1-T}\sum_{i=1}^{\infty}\frac{(-1)^{i-1}}{i}X_iT^i.
\]
Then $H(\mathbf{X})(T)=-\exp_T(G(\mathbf{X})(T))$ and
\[
\log_p(-(p-1)!)=G(\{Q(p,i)\}_i)(p).
\]
Since $|-(p-1)!-1|\le p^{-1}<p^{-1/(p-1)}$, we may apply the $p$-adic exponential and get
\begin{align*}
(p-1)!&=-\exp_p(\log_p(-(p-1)!))=-\exp_p(G(\{Q(p,i)\}_i)(p))=H(\{Q(p,i)\}_i)(p) \\
&=-1+\sum_{\nu=1}^{\infty}\frac{1}{\nu!}\widetilde\psi_\nu(Q(p,1),\dots,Q(p,\nu))p^\nu.
\end{align*}
Hence
\[
W_p=\sum_{\nu=1}^{\infty}\frac{1}{\nu!}\tilde{\psi}_\nu(Q(p,1),\dots,Q(p,\nu))p^{\nu-1}.
\]

Moreover,
\begin{align*}
G(\mathbf{X})(T)&=-\frac{1}{1-T}\sum_{i=1}^{\infty}\frac{(-1)^{i-1}}{i}X_iT^i \\
&=-\left(\sum_{j=0}^{\infty}T^j\right)\left(\sum_{i=1}^{\infty}\frac{(-1)^{i-1}}{i}X_iT^i\right) \\
&=-\sum_{k=1}^{\infty}\left(\sum_{i=1}^{k}\frac{(-1)^{i-1}}{i}X_i\right)T^k \\
&=\sum_{k=1}^{\infty}\frac{1}{k!}\left(\sum_{i=1}^{k}\frac{(-1)^ik!}{i}X_i\right)T^k.
\end{align*}
Since the complete Bell polynomials are characterized by
\[
\text{exp}\left(\sum_{i=1}^{\infty} \frac{1}{i!}X_iT^i\right)=1+\sum_{i=1}^{\infty} \frac{1}{i!}\mathbb{B}_i(X_1,\dots,X_i)T^i,
\]
then
\begin{align*}
H(\mathbf{X})(T)&=-\exp_T(G(\mathbf{X})(T)) \\
&=-1-\sum_{\nu=1}^{\infty}\frac{1}{\nu!}\mathbb{B}_{\nu}\left(\left\{\sum_{i=1}^{k}\frac{(-1)^ik!}{i}X_i\right\}_{k=1,\dots,\nu}\right)T^\nu.
\end{align*}
Therefore,
\[
\tilde{\psi}_\nu(\mathbf{X})=-\mathbb{B}_\nu\left(\left\{\sum_{i=1}^{k}\frac{(-1)^ik!}{i}X_i\right\}_{k=1,\dots,\nu}\right)\in \mathbb{Z}[X_1,\dots,X_\nu].
\]

We now expand $H(\mathbf{X})(T)$ without using the complete Bell polynomials. We have
\begin{align*}
&H(\mathbf{X})=-\text{exp}\left(\sum_{\ell=1}^{\infty} \frac{1}{\ell!}\left(\sum_{i=1}^{\ell} \frac{(-1)^i\ell!}{i}X_i\right)T^\ell\right)=-1-\sum_{j=1}^{\infty} \frac{1}{j!}\left(\sum_{\ell=1}^{\infty} \left(\sum_{i=1}^{\ell} \frac{(-1)^i}{i}X_i\right)T^{\ell}\right)^j \\
=&-1-\sum_{j=1}^{\infty} \frac{1}{j!}\sum_{\nu=j}^{\infty} \sum_{\substack{\ell_1,\dots,\ell_j\ge 1 \\ \ell_1+\dots+\ell_j=\nu}}\left(\prod_{s=1}^j \sum_{i=1}^{\ell_s} \frac{(-1)^i}{i}X_i\right)T^{\nu} \\
=&-1-\sum_{\nu=1}^{\infty} \left\{\sum_{j=1}^{\nu} \frac{1}{j!}\sum_{\substack{\ell_1,\dots,\ell_j\ge 1 \\ \ell_1+\dots+\ell_j=\nu}}\left(\prod_{s=1}^j \sum_{i=1}^{\ell_s} \frac{(-1)^i}{i}X_i\right)\right\}T^{\nu} \\
=&-1-\sum_{\nu=1}^{\infty} \left\{\sum_{j=1}^{\nu} \sum_{\substack{\ell_1,\dots,\ell_j\ge 1 \\ \ell_1+\dots+\ell_j=\nu}} \sum_{\substack{1\le i_1\le \ell_1 \\ \vdots \\ 1\le i_j\le \ell_j}} \frac{1}{j!}\prod_{s=1}^j \frac{(-1)^{i_s}}{i_s}X_{i_s}\right\}T^{\nu}.
\end{align*}
For fixed $i_1,\dots,i_j$, the number of tuples $(\ell_1,\dots,\ell_j)$ satisfying
\[
\ell_s\ge i_s,~~\ell_1+\dots+\ell_j=\nu
\]
is $\ds \left(\begin{array}{c} \nu-(i_1+\dots+i_j)+j-1 \\ j-1\end{array}\right)$. Hence
\begin{align*}
&H(\mathbf{X})=-1-\sum_{\nu=1}^{\infty} \left\{\sum_{j=1}^{\nu} \sum_{\substack{\ell_1,\dots,\ell_j\ge 1 \\ \ell_1+\dots+\ell_j=\nu}} \sum_{\substack{1\le i_1\le \ell_1 \\ \vdots \\ 1\le i_j\le \ell_j}} \frac{1}{j!}\prod_{s=1}^j \frac{(-1)^{i_s}}{i_s}X_{i_s}\right\}T^{\nu} \\
=&-1+\sum_{\nu=1}^{\infty} \left\{\sum_{j=1}^{\nu} \sum_{\substack{i_1,\dots,i_j\ge 1 \\ i_1+\dots+i_j\le \nu}} (-1)^{i_1+\dots+i_j+1}\left(\begin{array}{c} \nu-(i_1+\dots+i_j)+j-1 \\ j-1\end{array}\right) \frac{1}{j!}\prod_{s=1}^j \frac{1}{i_s}X_{i_s}\right\}T^{\nu}.
\end{align*}
Grouping the ordered tuples $(i_1,\dots,i_j)$ according to their multiplicities, write
\[
m_r=\#\{s~|~i_s=r\}~~(1\le r\le \nu).
\]
Then
\[
j=m_1+\dots+m_{\nu},~~i_1+\dots+i_j=m_1+2m_2+\dots+\nu m_{\nu}.
\]
The number of ordered tuples with fixed multiplicities $(m_1,\dots,m_{\nu})$ is $\ds \frac{(m_1+\dots+m_{\nu})!}{m_1!\dots m_{\nu}!}$. Thus
\begin{align*}
&H(\mathbf{X}) \\
=&-1+\sum_{\nu=1}^{\infty} \left\{\sum_{j=1}^{\nu} \sum_{\substack{i_1,\dots,i_j\ge 1 \\ i_1+\dots+i_j\le \nu}} (-1)^{i_1+\dots+i_j+1}\left(\begin{array}{c} \nu-(i_1+\dots+i_j)+j-1 \\ j-1\end{array}\right) \frac{1}{j!}\prod_{s=1}^j \frac{1}{i_s}X_{i_s}\right\}T^{\nu} \\
=&-1+\sum_{\nu=1}^{\infty} \Bigg\{\sum_{\substack{m_1,\dots,m_{\nu}\ge 0 \\ 1\le m_1+2m_2+\dots+\nu m_{\nu}\le \nu}} \frac{(m_1+\dots+m_{\nu})!}{m_1!\dots m_{\nu}!} \\
&~~~~~~~~~~~~~~~(-1)^{m_1+2m_2+\dots+\nu m_{\nu}+1}\left(\begin{array}{c} \nu-(m_1+2m_2+\dots+\nu m_{\nu})+(m_1+\dots+m_{\nu})-1 \\ (m_1+\dots+m_{\nu})-1\end{array}\right) \\
&~~~~~~~~~~~~~~~\frac{1}{(m_1+\dots+m_{\nu})!}\prod_{r=1}^{\nu} \frac{1}{r^{m_r}}X_r^{m_r}\Bigg\}T^{\nu} \\
=&-1+\sum_{\nu=1}^{\infty} \frac{1}{\nu!}\Bigg\{\sum_{\substack{m_1,\dots,m_{\nu}\ge 0 \\ 1\le m_1+2m_2+\dots+\nu m_{\nu}\le \nu}}(-1)^{m_1+2m_2+\dots+\nu m_{\nu}+1} \\
&~~~~~~~~~~~~~~~~~~\nu!\left(\begin{array}{c} \nu-(m_1+2m_2+\dots+\nu m_{\nu})+(m_1+\dots+m_{\nu})-1 \\ (m_1+\dots+m_{\nu})-1\end{array}\right) \\
&~~~~~~~~~~~~~~~~~~\left(\prod_{r=1}^{\nu} \frac{1}{r^{m_r}m_r!}\right)\left(\prod_{r=1}^{\nu}X_r^{m_r}\right)\Bigg\}T^{\nu}.
\end{align*}
Comparing coefficients with $\ds H(\mathbf{X})(T)=-1+\sum_{\nu=1}^{\infty} \frac{1}{\nu!}\tilde{\psi}_{\nu}(\mathbf{X})T^{\nu}$, we get
\begin{align*}
\tilde{\psi}_{\nu}(\mathbf{X})&=\sum_{\substack{m_1,\dots,m_{\nu}\ge 0 \\ 1\le m_1+2m_2+\dots+\nu m_{\nu}\le \nu}}(-1)^{m_1+2m_2+\dots+\nu m_{\nu}+1} \\
&~~~~~~~~~~~~\nu!\left(\begin{array}{c} \nu-(m_1+2m_2+\dots+\nu m_{\nu})+(m_1+\dots+m_{\nu})-1 \\ (m_1+\dots+m_{\nu})-1\end{array}\right) \\
&~~~~~~~~~~~~~~~~~~~~~~~~~~~~~~~~~~~~~~~~~~~~~~~~~~~~~~~~~~~~\left(\prod_{r=1}^{\nu} \frac{1}{r^{m_r}m_r!}\right)\left(\prod_{r=1}^{\nu}X_r^{m_r}\right).
\end{align*}
\\
\end{proof}

\section{A resolution of Kellner's conjectures}

\begin{cor} \n
Both of Kellner's conjectures are true.
\\
\end{cor}

\begin{proof} \n
By Theorem 1.3 and Corollary 1.7,
\begin{align*}
\psi_{\nu}(\mathbf{X})&=\sum_{\substack{m_1,\dots,m_{\nu}\ge 0 \\ 1\le m_1+2m_2+\dots+\nu m_{\nu}\le \nu}}(-1)^{m_1+2m_2+\dots+\nu m_{\nu}+1} \\
&~~~~~~~~~~~~\nu!\left(\begin{array}{c} \nu-(m_1+2m_2+\dots+\nu m_{\nu})+(m_1+\dots+m_{\nu})-1 \\ (m_1+\dots+m_{\nu})-1\end{array}\right) \\
&~~~~~~~~~~~~~~~~~~~~~~~~~~~~~~~~~~~~~~~~~~~~~~~~~~~~~~~~~~~~\left(\prod_{r=1}^{\nu} \frac{1}{r^{m_r}m_r!}\right)\left(\prod_{r=1}^{\nu}X_r^{m_r}\right).
\end{align*}
Therefore, the monomials appearing in $\psi_{\nu}$ are in bijection with the non-empty partitions of integers $m$ with $1\le m\le \nu$. Hence the number of terms of $\psi_{\nu}$ is $\mathcal{P}_{\Sigma}(\nu)$, so that Conjecture 1.2 (i) is true.

We have
\begin{align*}
\sum_{\nu=1}^{\infty}\frac{1}{\nu!}\psi_\nu(\{\pm1\})T^\nu&=H(\{\pm 1\})+1=1-\text{exp}\left(\pm \sum_{i=1}^{\infty} \frac{(-1)^i}{i}\frac{T^i}{1-T}\right) \\
&=1-\text{exp}\left(\left(\mp \frac{1}{1-T}\right)\log(1+T)\right)=1-(1+T)^{\mp \frac{1}{1-T}}.
\end{align*}

Since
\[
\psi_{\nu}(\mathbf{X})=-\mathbb{B}_\nu\left(\left\{\sum_{i=1}^{k}\frac{(-1)^ik!}{i}X_i\right\}_{k=1,\dots,\nu}\right),
\]
by Theorem 1.3, after specializing all variables $X_i$ to $+1$ or all variables $X_i$ to $-1$, respectively, we obtain
\[
\psi_\nu(\{\pm1\})=-\mathbb{B}_\nu(\{\mp k!\overline{H}_k\}_{k=1,\dots,\nu}).
\]
Hence, Conjecture 1.2 (ii) is also true.
\\
\end{proof}


\begin{thebibliography}{9}

\bibitem{Kellner}
B.~C.~Kellner,
\textit{Wilson's theorem modulo higher prime powers I: Fermat and Wilson quotients},
arXiv:2509.05235 [math.NT], 2025.
\newblock DOI: \texttt{10.48550/arXiv.2509.05235}.

\bibitem{Lerch}
M.~Lerch,
\textit{Zur Theorie des Fermatschen Quotienten \(\frac{a^{p-1}-1}{p}=q(a)\)},
Math. Ann. \textbf{60} (1905), 471--490.
\newblock DOI: \texttt{10.1007/BF01561092}.

\end{thebibliography}
\end{document}